\documentclass[preprint,12pt]{elsarticle}

\usepackage{multirow}
\usepackage{amsmath,fancyhdr}
\usepackage{amssymb}
\usepackage{array}
\usepackage{graphicx}
\usepackage{hyperref}
\usepackage{rotating}
\usepackage{epsf}
\usepackage{url}
\usepackage{epsfig}
\usepackage{graphicx,color}
\newcommand{\ignore}[1]{}

\newtheorem{theorem}{Theorem}%[section]
\newtheorem{lemma}{Lemma}%[section]
\newtheorem{corollary}{Corollary}%[section]

\usepackage{tikz}

\begin{document}

\begin{frontmatter}

%\title{Fractional model for CoViD-19 dynamics}
\title{A fractional-order model for CoViD-19 dynamics with reinfection and the importance of quarantine}

\author{Jo\~ao P. S. Maur\'icio de Carvalho$^{\star}$\let\thefootnote\relax\footnote{$^\star$ corresponding author}$^a$ and Beatriz Moreira-Pinto$^{b}$\\
 \medskip
  $^a$Faculty of Sciences, University of Porto, \\
 Rua do Campo Alegre s/n, 4169-007 Porto, Portugal\\
 \textit{up200902671@fc.up.pt}\\
 \medskip
  $^b$UCIBIO, REQUIMTE, Faculty of Pharmacy, University of Porto \\
 Rua de Jorge Viterdo Ferreira, 228, 4050-313 Porto, Portugal \\
 \textit{abeatriz\_pinto@hotmail.com}\\
}

\begin{abstract}
Coronavirus disease 2019 (CoViD-19) is an infectious disease caused by severe acute respiratory syndrome coronavirus 2 (SARS-CoV-2). Among many symptoms, cough, fever and tiredness are the most common. People over 60 years old and with associated comorbidities are most likely to develop a worsening health condition. This paper proposes a non-integer order model to describe the dynamics of CoViD-19 in a standard population. The model incorporates the reinfection rate in the individuals recovered from the disease. Numerical simulations are performed for different values of the order of the fractional derivative and of reinfection rate. The results are discussed from a biological point of view.   
\end{abstract}

\begin{keyword}
COVID-19, reinfection, mathematical model, epidemic model, fractional calculus
\end{keyword}

\end{frontmatter}

\section{Introduction}

At the end of the year 2019, a newly discovered coronavirus named Severe Acute Respiratory Syndrome Coronavirus 2 (SARS-CoV-2) emerged in Wuhan, China \cite{Bavel,Rev3_1}. The disease was later designated by the World Health Organization (WHO) as Coronavirus Disease 2019 or CoViD-19. CoViD-19 has demonstrated a great capacity of propagation directly through human-to-human contact and the epidemic quickly began to spread on a worldwide level, claiming multiple lives throughout its course \cite{Dao}. At the time of writing, worldwide CoViD-19 cases surpassed 185 million and caused more than 4 million deaths \cite{who}. Although infected individuals may be asymptomatic, the most common mild symptoms range from coughing and wheezing to a lack of smell and taste \cite{Li}. In people over 65 years old and/or with pre-existing conditions, such as diabetes, obesity and hypertension, the risk of worsening symptoms is higher \cite{Rothan, Bompard, Rodriguez, Wang}. Respiratory, neurological and hepatic diseases are the most serious diseases caused by SARS-CoV-2 \cite{Rev3_3}.

With the increasing number of cases and deaths, strategies have been adopted to slow down the spread of the virus like social distancing, the use of face masks and isolation of infected people \cite{Berger}.

Several mathematical models have been proposed to understand the dynamics of CoViD-19. These are extremely valuable to comprehend CoViD-19 mechanism of transmission, as well as for predicting disease behavior and controlling possible outbreaks \cite{Jewell, Thomas, Zeb}. \c{C}akan \cite{akan} proposed a mathematical SEIR epidemic model to evaluate the impact of CoViD-19 in a hospital environment. The author concluded that an increase in contact rates between susceptible and infected individuals may lead to hospitals breakdown by depleting their resources. Buonomo \cite{springer_reinfection} proposed a mathematical SIRI model to analyze the effects of a vaccine on a population where CoViD-19 was predominant, suggesting that the incidence of the disease can be reduced by correct and quick information given to the population. Khoshnaw {\it el al}.~\cite{Khoshnaw} studied the sensitivity of important parameters in the reproduction number variation. The results show that the contact rate, the exposure rate during quarantine and the transition rate of exposed individuals play a key role in the spread of the disease. 

%\subsection*{Fractional calculus}

Non-integer order calculus, known as fractional calculus (FC), generalizes integral and differential calculus. Briefly, fractional order (FO) operators can be representative of memory functions, making the dynamics of variables more realistic \cite{JPCarvalho,Rev3_2, Rev3_5}. Recent studies on CoViD-19 have incorporated FC \cite{Shah1, Shah2, Shah3, Shah4}. Epidemiological and biomathematical models are two of several areas where FC is being applied \cite{Samko,Rev3_4,Rev3_6}. Ahmad {\it et al}.~\cite{Ahmad} performed simulations of a fractional model for CoViD-19 transmission considering different values of the non-integer order derivative $\alpha$, and concluding that $\alpha = 0.97$ is the value that best fits the real data. Also, Zhang {\it et al}.~\cite{springer} developed a non-integer order model for CoViD-19 dynamics. The authors analyzed the reproduction number and investigated the asymptotic stability of the proposed model. The numerical simulations indicate that there is good agreement between the theoretical and the numerical results.

These models then encouraged us to formulate a FO mathematical model for population dynamics in the presence of CoViD-19 and analyze the impact of isolation, reinfection and recovery rates of the individuals. The main goal of this work is to understand how FC influences the dynamics of populations over time and try to understand how quarantine can be so important in reducing the number of CoViD-19 cases. Our results could help policy makers to devise strategies to reduce heavy economic and social burden of SARS-CoV-2 infection in the world.

In Section \ref{descricao} we describe the model and prove that it is positive and bounded. In Section \ref{R0_SFE} we calculate the basic reproduction number, study the stability around the disease-free equilibrium point and perform the sensitivity analysis of relevant parameters in the spread of CoViD-19. In Section \ref{RESULTS} we simulate the model for all relevant parameters and we comment on their results. We draw some conclusions and present future work in Section \ref{conclusions}.

% MODELO
\section{Model interpretation} \label{descricao}
We adapted and adjusted the \emph{SIQR} model (susceptible -- infected -- quarentined -- recovered) of the authors of \cite{ZhienMa2009}. Four classes of individuals incorporate the model: susceptible, $S(t)$, infected, $I(t)$, isolated/quarantined, $Q(t)$, and recovered, $R(t)$. With respect to our model, we define $\Omega = \{ (\lambda,\beta,\mu,r,\sigma,\theta) \in (\mathbb{R}^+)^6 \}$ as the set of parameters. The recruitment rate of susceptible individuals is given by $\lambda^{\alpha}$. The contact rate between susceptible and recovered individuals and infected ones is given by $\beta^{\alpha}$. After contact, susceptible and recovered individuals move into the infected class. The term $\sigma^{\alpha} I$ represents the fraction of infected individuals who became isolated. Isolated individuals recover from the disease at a rate $\theta^{\alpha}$. The susceptibility of a recovered individual to be reinfected is given by $p^{\alpha}$ \cite{springer_reinfection}. Thus, we consider that $r^{\alpha} = \beta^{\alpha} p^{\alpha}$ is the reinfection rate of individuals who have already recovered from the disease. Parameters $\mu_S^{\alpha}$, $\mu_I^{\alpha}$, $\mu_Q^{\alpha}$ and $\mu_R^{\alpha}$ are the natural death rates of susceptible, exposed, infected, isolated and recovered individuals, respectively. It is assumed an equal value for every natural death rate to simplify algebraic calculations, \emph{i.e.} $\mu_S^{\alpha} = \mu_I^{\alpha} = \mu_Q^{\alpha} = \mu_R^{\alpha} \equiv \mu^{\alpha}$. Figure \ref{boxes} illustrates the interaction between the classes of susceptible, infected, quarentined and recovered individuals in model \eqref{modelo}. A description of the model variables and all parameters can be found in Table \ref{variables} and Table \ref{tabela}, respectively.

The system of FO equations is given by

\begin{equation}\label{modelo}
\begin{array}{lcl}
\dfrac{d^{\alpha}S}{dt^{\alpha}} = \lambda^{\alpha} - \beta^{\alpha}SI - \mu^{\alpha} S \\
\\
\dfrac{d^{\alpha}I}{dt^{\alpha}} = \beta^{\alpha}SI + r^{\alpha} RI - \sigma^{\alpha}I - \mu^{\alpha}I \\
\\
\dfrac{d^{\alpha}Q}{dt^{\alpha}} = \sigma^{\alpha}I - \theta^{\alpha}Q - \mu^{\alpha}Q \\
\\
\dfrac{d^{\alpha}R}{dt^{\alpha}} = \theta^{\alpha}Q -r^{\alpha} RI - \mu^{\alpha}R,
\end{array}
\end{equation}

\medskip

\noindent where $\alpha \in (0, 1]$ is the order of the fractional derivative. We use the concept of a FO derivative proposed by Caputo:

\begin{equation}
\label{caputo}
\begin{array}{lcl}
\dfrac{d^{\alpha}y(t)}{{dt}^{\alpha}}  =  I^{p-\alpha} y^{(p)} (t), \,\,\,\,\, t>0,
\end{array}
\end{equation}

\medskip

\noindent where $p=[\alpha]$ is the integer part of $\alpha$, $y^{(p)}$ is the $p$\,-th derivative of $y(r)$ and $I^{p_1}$ is the Riemann-Liouville fractional integral (see \cite{Samko} and references therein)

\begin{equation}
\label{liouville}
\begin{array}{lcl}
I^{p_1} z(t)  =  \dfrac{1}{\Gamma(p_1)} \displaystyle \int_0^t (t - t')^{p_1 - 1} z(t') dt'.
\end{array}
\end{equation}

\begin{table}[!h]
\caption{Description of the variables of model (\ref{modelo})}
\medskip
\centering
\def\arraystretch{1.2} 
\begin{tabular}{  lr}
\hline
\textbf{Variable} & \textbf{Symbol} \\
\hline
Susceptible population & $S(t)$   \\
Infected population & $I(t)$  \\
Isolated population & $Q(t)$  \\
Recovered population & $R(t)$  \\
\hline
\end{tabular}
\label{variables}
\end{table}

\usetikzlibrary{arrows,positioning}
\begin{center}
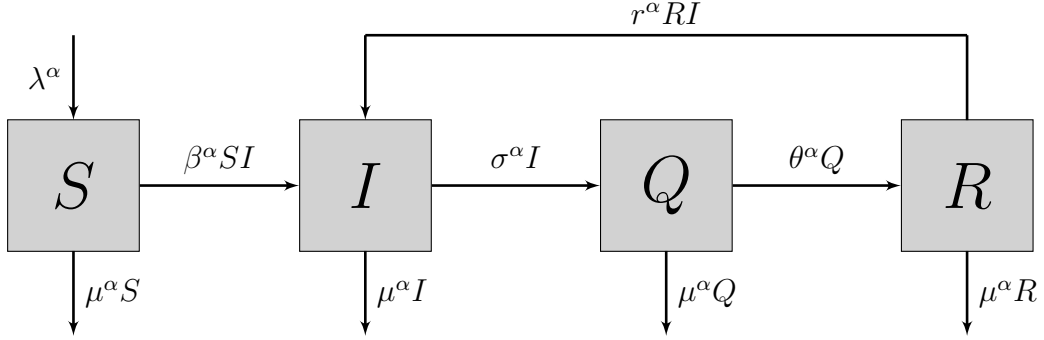
\begin{figure}[ht!]
\begin{tikzpicture}
[
%node distance=2.5cm,
%rounded corners,
auto,
>=latex',
every node/.append style={align=center},
int/.style={draw, minimum size=1.75cm}
]

    \node [fill=lightgray!70,int] (S)             {\Huge $S$};
    \node [fill=lightgray!70,int, right=3cm] (I) {\Huge $I$};
    \node [fill=lightgray!70,int, right=7cm] (Q) {\Huge $Q$};
    \node [fill=lightgray!70,int, right=11cm] (R) {\Huge $R$};
    
    \node [below=of S] (s) {} ;
    \node [below=of I] (i) {} ;
    \node [below=of Q] (q) {};
     \node [below=of R] (r) {};
    
    \node [above=of S] (os) {};
    \node [above=of I] (oi) {};
    \node [above=of Q] (oq) {};
    \node [above=of R] (or) {};
    
    \coordinate[right=of I] (out);
    \path[->, auto=false,line width=0.35mm]
    			(S) edge node {$\beta^{\alpha}SI$ \\[.6em]} (I)
                          (I) edge node {$\sigma^{\alpha}I$       \\[.6em] } (Q) 
                            (Q) edge node {$\theta^{\alpha}Q$       \\[.6em] } (R) ;                     
                      %    (S) edge  [out=-120, in=-60] node[below] {$S(A-S)$ \\ [0.2em] \emph{Logistic map}} (S);

    \path[->, auto=false,line width=0.35mm]
    			(S) (0,-0.87cm) edge [] node[right]{$\mu^{\alpha} S$} (0,-2cm) (s)
			(os) (0,2cm) edge [] node[left]{$\lambda^{\alpha}$} (0,0.87cm) (S) 
    			(I) (3.88,-0.87cm) edge [] node[right]{$\mu^{\alpha}  I$} (3.88,-2cm) (i) 
    			(Q) (11.88,-0.87cm) edge [] node[right]{$\mu^{\alpha} R$} (11.88,-2cm) (q) 
			(Q) (7.88,-0.87cm) edge [] node[right]{$\mu^{\alpha}  Q$} (7.88,-2cm) (q);

    \path[-, auto=false,line width=0.35mm]
    			(R) (3.88,0.87cm) edge [<-] node[right]{} (3.88,2cm) (or)
			(or) (3.88,2cm) edge [] node[above]{$r^{\alpha}RI$} (11.88,2cm) (oi)
			(oi) (11.88,2cm) edge [] node[below]{} (11.88,0.87cm) (I) ;

    % \path[-, auto=false,line width=0.35mm] (S) edge [] node[right]{} (os) ;
     %\path[-, auto=false,line width=0.35mm] (os)  edge [] node[below]{$p_{\alpha}(t) A$}  (or) ;
     %\path[->, auto=false,line width=0.35mm] (or) edge[] node[right]{} (R);

\end{tikzpicture}
\caption{\small Schematic diagram of model \eqref{modelo}. Boxes represent compartments, and arrows indicate the flow between the compartments.}
\label{boxes}
\end{figure}
\end{center}

\subsection{Model properties analysis}

The solutions of the system (\ref{modelo}) remain non-negative for the entire domain, $t>0$. Let $\mathbb{R}_+^4 = \{x \in \mathbb{R}^4 \,\, | \,\, x \geq 0 \}$ and $x(t) = \left(S(t), I(t), Q(t), R(t) \right)^T$. First, we quote the following Generalized Mean Value Theorem \cite{GMVT} and corollary.

\begin{lemma}
\cite{GMVT} Suppose that $f(x) \in C[a,b]$ and $D_a^{\alpha}f(x) \in C(a,b]$, where $0<\alpha \leq 1$, thus

\begin{equation}
\label{lema1}
\begin{array}{lcl}
	f(x)  =  f(a)+\dfrac{1}{\Gamma(\alpha)}(D_a^{\alpha}f)(\xi) \cdot (x-a)^{\alpha}
\end{array}
\end{equation}

\medskip

\noindent for $a \leq \xi \leq x, \forall x \in (a,b]$ and $\Gamma(\cdot)$ is the gamma function.
\end{lemma}

\begin{corollary} \label{corolario}
Let $f(x) \in C[a,b]$ and $D_a^{\alpha}f(x) \in C(a,b]$, for $0<\alpha \leq 1$.

\begin{enumerate}
\item If $D_a^{\alpha}f(x) \geq 0$, $\forall x \in (a,b)$, then $f(x)$ is non-decreasing for each $x \in [a,b]$; \\
\item If $D_a^{\alpha}f(x) \leq 0$, $\forall x \in (a,b)$, then $f(x)$ is non-increasing for each $x \in [a,b]$.
\end{enumerate}
\end{corollary}

\noindent This proves the main theorem.

\begin{theorem}
There is a unique solution $x(t) = \left(S(t), I(t), Q(t), R(t) \right)^T$ to the system (\ref{modelo}) in the entire domain $(t \geq 0)$. Furthermore, the solution remains in $\mathbb{R}_+^4$.
\end{theorem}

\begin{proof}

As we can see from Theorem 3.1 and Remark 3.2 of \cite{lin2007}, the solution of the initial value problem exists and is unique, for $t \geq 0$. Then, it is enough to prove that the non-negative orthant $\mathbb{R}_+^4$ is positively invariant. So, we must demonstrate that the vector field points to $\mathbb{R}_+^4$ in each hyperplane, thus limiting the non-negative orthant. Hence, we have:

\begin{equation}\label{limitada}
\begin{array}{lcl}
\dfrac{d^{\alpha}S}{dt^{\alpha}} \,|_{S=0}   =  \lambda^{\alpha} \geq 0 \\
\\
\dfrac{d^{\alpha}I}{dt^{\alpha}} \,|_{I=0}  =  0 \\
\\
\dfrac{d^{\alpha}Q}{dt^{\alpha}} \,|_{Q=0}  =  \sigma^{\alpha}I  \geq 0 \\
\\
\dfrac{d^{\alpha}R}{dt^{\alpha}} \,|_{R=0}    =  \theta^{\alpha}Q \geq 0 .
\end{array}
\end{equation}
\end{proof}

\medskip

\noindent According to the Corollary \ref{corolario}, it can be concluded that the solution remains in $\mathbb{R}_+^4$.

\section{Reproduction number and disease-free equilibria} \label{R0_SFE}
In this section we compute the reproduction number, $\mathcal{R}_0$ of the model (\ref{modelo}). Basic reproduction number is the number of secondary infections caused by a single infected person in a susceptible population \cite{vandendriessch}.

A disease-free equilibrium of the model (\ref{modelo}) is obtained via imposing $I = Q = R = 0$. Then we get: 

\begin{equation*}
\label{eq_point}
%\scalemath{0.85}{
\begin{array}{lcl}
E^{\star} &=& (S^{\star}, I^{\star}, Q^{\star}, R^{\star}) \\
\\
&=& \left(\dfrac{\lambda^{\alpha}}{\mu^{\alpha}}, 0, 0, 0 \right) .
\end{array}
%}
\end{equation*}

\medskip

\noindent Using Lemma 1 of \cite{vandendriessch} in system (\ref{modelo}), the matrices for the new infection terms, $F$, and the remaining terms, $V$, are the following:

\begin{equation}\label{FV2}
\begin{array}{lcl}
F=\left(\begin{array}{cc}
\beta^{\alpha} S^{\star} + r^{\alpha} R^{\star} & \beta^{\alpha}I^{\star} \\ 
\\
0 & 0
\end{array}\right) \\
\\
V=\left(\begin{array}{cc}
\sigma^{\alpha} + \mu^{\alpha} & 0 \\ 
\\
\beta^{\alpha} S^{\star} & \beta^{\alpha} I^{\star} + \mu^{\alpha}
\end{array}\right) ,
\end{array}
\end{equation}

\medskip

\noindent and the associative basic reproduction number is given by:

\begin{equation}\label{R0}
\begin{array}{lcl}
\mathcal{R}_0 = \rho \left(FV^{-1}\right) = \dfrac{\beta^{\alpha} \lambda^{\alpha}}{\mu^{\alpha} (\sigma^{\alpha} + \mu^{\alpha}) } ,
\end{array}
\end{equation}

\medskip

\noindent where $\rho$ is the spectral radius of the matrix $FV^{-1}$. 

By Theorem 2 of \cite{vandendriessch} we obtain the Lemma \ref{estabilidadeR0}.

\begin{lemma}\label{estabilidadeR0}
The disease-free equilibrium $E^{\star}$ is locally asymptotically stable if $\mathcal{R}_0 < 1$ and unstable if $\mathcal{R}_0 > 1$.
\end{lemma}

\begin{proof}
Let 

\begin{equation}\label{linear}
\begin{array}{lcl}
\mathcal{L}=\left(\begin{array}{cccc}
-\beta^{\alpha} I - \mu^{\alpha} & -\beta^{\alpha} S & 0 & 0 \\ 
\\
\beta^{\alpha} I & \beta^{\alpha} S + r^{\alpha}R - \sigma^{\alpha}  - \mu^{\alpha}  & 0 & r^{\alpha} I \\
\\
0 & \sigma^{\alpha} & - \theta^{\alpha} - \mu^{\alpha}  & 0 \\
\\
0 & -r^{\alpha}R & \theta^{\alpha}  &  r^{\alpha} I - \mu^{\alpha} 
\end{array}\right) 
\end{array}
\end{equation}

\medskip

\noindent be the matrix of linearization of the model (\ref{modelo}). Therefore, the matrix $\mathcal{L}$ around the disease-free equilibrium $E^{\star}$, takes the form:

\begin{equation}\label{linear}
\begin{array}{lcl}
\mathcal{L}(E^{\star})=\left(\begin{array}{cccc}
-\mu^{\alpha} & -\dfrac{\beta^{\alpha} \lambda^{\alpha}}{\mu^{\alpha} } & 0 & 0 \\ 
\\
0 & \dfrac{\beta^{\alpha} \lambda^{\alpha}}{\mu^{\alpha}} - \sigma^{\alpha} - \mu^{\alpha} & 0 & 0 \\
\\
0 & \sigma^{\alpha} & - \theta^{\alpha} - \mu^{\alpha}  & 0 \\
\\
0 & 0 & \theta^{\alpha}  & - \mu^{\alpha} 
\end{array}\right) .
\end{array}
\end{equation}

\medskip

\noindent The eigenvalues of $\mathcal{L}(E^{\star})$ are given by:

\begin{eqnarray}
\label{E3_stable}
\nonumber \lambda_1 &=& - \theta^{\alpha} - \mu^{\alpha} ,
\nonumber \\ 
\nonumber \\
\lambda_2  &=& \dfrac{ \beta^{\alpha} \lambda^{\alpha} - {\mu^{\alpha}}^2- \mu^{\alpha} \sigma^{\alpha}}{\mu^{\alpha}} ,\\
\nonumber \\
\nonumber \lambda_3 &=& \lambda_4 \,\,\,\, = \,\,\,\, -\mu^{\alpha} .
\end{eqnarray}

\medskip

\noindent It is easy to verify that the eigenvalues $\lambda_1$, $\lambda_3$ and $\lambda_4$ have negative real part. With regard to $\lambda_2$, there is a negative real part if

\begin{eqnarray}
\label{lambda2_stable}
\nonumber &&  \dfrac{ \beta^{\alpha} \lambda^{\alpha} - {\mu^{\alpha}}^2- \mu^{\alpha} \sigma^{\alpha}}{\mu^{\alpha}} < 0 \\
\nonumber \\
\Leftrightarrow && \nonumber \beta^{\alpha} \lambda^{\alpha} - {\mu^{\alpha}}^2 - \mu^{\alpha} \sigma^{\alpha} < 0 \\
\nonumber \\
\nonumber \Leftrightarrow&&   \beta^{\alpha} \lambda^{\alpha} < \mu^{\alpha}(\sigma^{\alpha} + \mu^{\alpha}) \\
\nonumber \\
\nonumber \Leftrightarrow &&   \dfrac{\beta^{\alpha} \lambda^{\alpha}}{\mu^{\alpha}(\sigma^{\alpha} + \mu^{\alpha})} < 1 \\
\nonumber \\
\nonumber \overset{(\ref{R0})}{\Leftrightarrow} &&   \mathcal{R}_0 < 1 .
\end{eqnarray}

\medskip

\noindent Thus, if $\mathcal{R}_0 < 1$, then all eigenvalues have negative real part. Therefore $E^{\star}$ is locally asymptotically stable under this condition. On the other hand, if 

\begin{eqnarray}
\label{lambda2_unstable}
\nonumber&&  \lambda_2 > 0 \\ 
\nonumber \\
\nonumber \Leftrightarrow &&   \dfrac{ \beta^{\alpha} \lambda^{\alpha} - {\mu^{\alpha}}^2- \mu^{\alpha} \sigma^{\alpha}}{\mu^{\alpha}} > 0 \\
\nonumber \\
\nonumber \Leftrightarrow &&   \beta^{\alpha} \lambda^{\alpha} - {\mu^{\alpha}}^2 - \mu^{\alpha} \sigma^{\alpha} > 0 \\
\nonumber \\
\nonumber \Leftrightarrow &&   \beta^{\alpha} \lambda^{\alpha} > \mu^{\alpha}(\sigma^{\alpha} + \mu^{\alpha}) \\
\nonumber \\
\nonumber \Leftrightarrow &&   \dfrac{\beta^{\alpha} \lambda^{\alpha}}{\mu^{\alpha}(\sigma^{\alpha} + \mu^{\alpha})} > 1 \\
\nonumber \\
\nonumber \overset{(\ref{R0})}{\Leftrightarrow} &&  \mathcal{R}_0 > 1 ,
\end{eqnarray}

\medskip

\noindent then $\lambda_2 > 0$. Therefore $E^{\star}$ is unstable.
\end{proof}

\subsection{Sensitivity analysis}

Sensitivity indices allow us to have a perspective on the relative change of a variable when a parameter varies. This sensitivity index is the ratio between the relative change in the variable and the relative change in the parameter. When the variable, $v$, is a differentiable function of these parameters, $p$, the sensitivity index can be calculated through partial derivatives, using the following expression \cite{chitnis2008}:

\begin{equation}\label{sens_formula}
%\scalemath{0.80}{
	\begin{array}{lcl}
	\gamma_p^v = \dfrac{\partial v}{\partial p} \times \dfrac{p}{v} .
\end{array}
%	}
\end{equation}

\medskip

\noindent In the case of $\mathcal{R}_0$, comes that

\begin{equation}\label{RR}
%\scalemath{0.80}{
	\begin{array}{lcl}
	\gamma_p^{\mathcal{R}_0} = \dfrac{\partial \mathcal{R}_0}{\partial p} \times \dfrac{p}{\mathcal{R}_0} ,
\end{array}
%	}
\end{equation}

\medskip

\noindent From (\ref{R0}) and (\ref{RR}) we analyzed how the sensitivity indices $\lambda^{\alpha}$, $\beta^{\alpha}$ and $\sigma^{\alpha}$ influence the basic reproduction number. We found that the sensitivity indices do not depend on the values of parameters $\lambda^{\alpha}$ and $\beta^{\alpha}$. However, we noticed that it depends on the value of parameter $\sigma^{\alpha}$. To perform this calculation, we use $\sigma^{\alpha} = 1.69 \times 10^{-2}$ (see Table \ref{tabela}) and we obtained the results given in Table \ref{indices}.

%\begin{eqnarray}
%\label{indicies_parametros}
%\nonumber \gamma_{\lambda^{\alpha}}^{\mathcal{R}_0} = 1 \quad , \qquad \gamma_{\beta^{\alpha}}^{\mathcal{R}_0}  = 1 \qquad \text{and} \qquad  \gamma_{\sigma^{\alpha}}^{\mathcal{R}_0} & \simeq &  - 0.98 \qquad \text{(see Table \ref{indices})}.
%\end{eqnarray}

\begin{table}[!h]
\caption{Sensitivity indices for relevant parameters of model (\ref{modelo}). Parameter values are given in the Table \ref{tabela}}
\medskip
\centering
\def\arraystretch{1.8} 
\begin{tabular}{  c lcl c|}
\hline
\textbf{Index} & & \textbf{Sensitivity index sign} \\
\hline

$\gamma_{\lambda^{\alpha}}^{\mathcal{R}_0}$ && $+ 1.00$  \\ 

$\gamma_{\beta^{\alpha}}^{\mathcal{R}_0}$ && $+ 1.00$   \\

$\gamma_{\sigma^{\alpha}}^{\mathcal{R}_0} $ && $- 0.98$  \\
\hline
\end{tabular}
\label{indices}
\end{table}

The sensitivity indices signs of Table \ref{indices} give us information about the variation of the value of $\mathcal{R}_0$. We concluded that the parameter $\lambda^{\alpha}$ and $\beta^{\alpha}$ contribute to the spread of the disease. This means that when the values of recruitment and contact rate between susceptible/recovered and infected individuals increase, the number of infected people also increases. The magnitude of $\lambda^{\alpha}$ and $\beta^{\alpha}$ is positive and of equal value. On the other hand, the rate at which infected people are quarantined, $\sigma^{\alpha}$, has an opposite effect. Thus, the isolation rate slows the spread of the disease.

\section{Numerical results} \label{RESULTS}

We simulate the model (\ref{modelo}) for distinct values of the order of the fractional derivative, $\alpha \in [0, 1]$, and for biologically relevant parameters. We apply the predictor-corrector PECE method of Adams-Bashford-Moulton type \cite{FOsubroutine}. We use epidemiological parameter values based on early estimation of novel coronavirus CoViD-19 provided in \cite{springer_reinfection}, \cite{springer} and \cite{Read2020} (please see Table \ref{tabela}). We then assume that they are valid from the context of our work. The initial conditions are:

\begin{eqnarray}
\label{ini_cond}
S(0) = 153 \,\,, \quad I(0) = 138 \,\,, \quad Q(0) = 68 \quad \text{and} \quad R(0) = 20, 
\end{eqnarray}

\medskip

\noindent provided in \cite{springer}. 

\begin{table}[!h]
\caption{Parameter values used in numerical simulations of model (\ref{modelo})}
\medskip
\centering
%{\footnotesize
\scalebox{0.85}{
\def\arraystretch{1.3}
\begin{tabular}{ l c lc  cl }
\hline
\textbf{Parameter} & \textbf{Symbol} & \textbf{Value} & \textbf{Reference}  \\
\hline
Recruitment rate of susceptible individuals & $\lambda^{\alpha}$ & $1.45 \times 10^{-1}$  & \cite{springer}  \\
Contact rate with infected individuals & $\beta^{\alpha}$ & $3.80 \times 10^{-4}$  & \cite{springer} \\
Isolation rate of infected individuals & $\sigma^{\alpha}$ & $1.69 \times 10^{-2}$ & \cite{springer}  \\
Recovery rate of isolated individuals & $\theta^{\alpha}$ & $1.81 \times 10^{-2}$  & \cite{springer} \\
Susceptibility due to previous infection & $p^{\alpha}$ & $[0, 1)$ & \cite{springer_reinfection}, \cite{Read2020}  \\
Natural death rate of individuals & $\mu^{\alpha}$ & $4.10 \times 10^{-4}$ & \cite{springer} \\
Reinfection rate of recovered individuals & $r^{\alpha}$ & $[0, 3.80 \times 10^{-4})$ & \cite{springer_reinfection}, \cite{Read2020} \\
\hline
\end{tabular}
}
%\medskip
\label{tabela}
%}
\end{table}
Figure \ref{comportamento} shows the behaviour of population classes of model (\ref{modelo}) for $\alpha = \{0.96, 0.98, 1\}$. Over time, the density of infected people reaches its peak (around 200 people) after approximately 25 days. However, the number of infected individuals tends to decrease to less than 100. One of the factors that may influence this decrease is the increase in the number of isolated and recovered people. This happens regardless of the value of $\alpha$. Furthermore, all three simulations show the same asymptotic behavior.

In Figures \ref{sigma} and \ref{r} we varied the values of the parameters $\sigma^{\alpha}$ and $r^{\alpha}$ within the same order of magnitude to analyze how they affect individuals' behavior in each class. In Figure \ref{sigma}, the dynamics of infected and isolated individuals were simulated for three isolation rates, $\sigma^{\alpha}$, considering different $\alpha$ values. In the first days, it is observed that the number of infected people decreases with the increase of people in quarantine. Moreover, the higher the isolation rate, the greater the decrease of infected individuals. The number of CoViD-19 positive is controlled through the isolation of confirmed cases, preventing the spread of the disease. This causes the number of infected people to decrease in the long run. As a result, the number of isolated ones slowly decreases. This happens regardless of the value of $\alpha$. Furthermore, the lower the value of the order of the fractional derivative, the lower the number of infected people.

Figure \ref{r} displays the behavior of infected and isolated individuals considering three values of the reinfection rate, $r^{\alpha}$. Increasing the reinfection rate increases the number of sick people, regardless of the value of $\alpha$. With the increase of infected individuals, the number of people who become isolated also increases. This behavior is independent of the value of the derivative of FO. In addition, the lower the value of $\alpha$ the lower the number of sick people and people in isolation. 

Figure \ref{contour} describes the density of infected and recovered individuals in the first 1000 days, for different combinations of isolation and recovery rates, $\sigma^{\alpha}$ and $\theta^{\alpha}$ respectively, and $\alpha = 1$. On the left, low recovery rate promotes a higher number of CoViD-19 patients (around 350 people). In general, higher values of $\theta^{\alpha}$ mean fewer people are infected. On the right, high recovery rates combined with values greater than $10^{-2}$ of isolation rate promotes a relatively high number of people recovered from the disease (greater than 140 people).

\begin{figure}[!]
%\vspace{2.5cm}
\center
\includegraphics[scale=0.40]{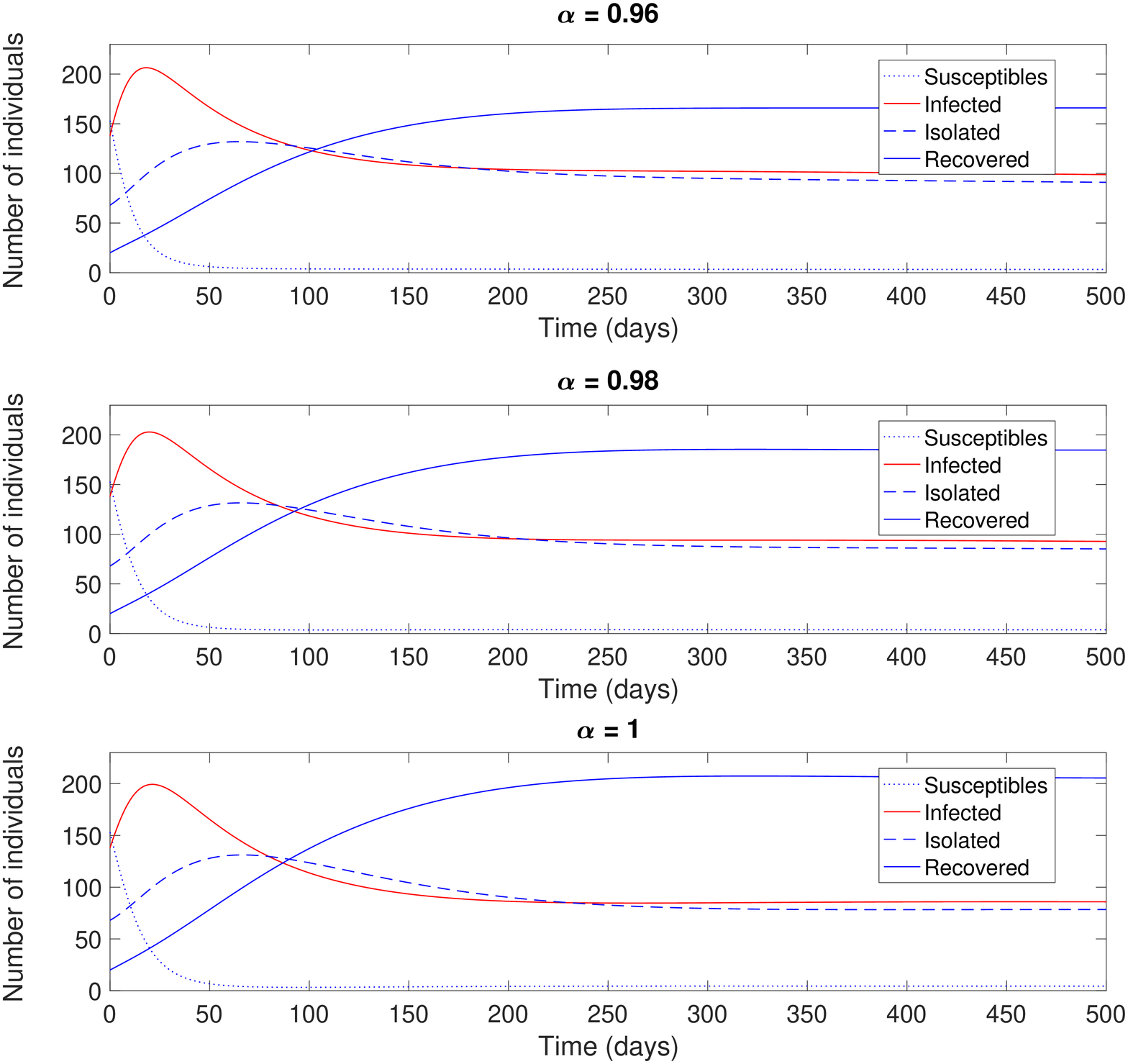}
\caption{Dynamics of model (\ref{modelo}) for $\alpha = \{0.96, 0.98, 1\}$. Initial conditions and parameter values are given in \eqref{ini_cond} and in the Table \ref{tabela}, respectively.}
\label{comportamento}
\end{figure}

\begin{figure}[!]
%\vspace{2.5cm}
\center
\includegraphics[scale=0.29]{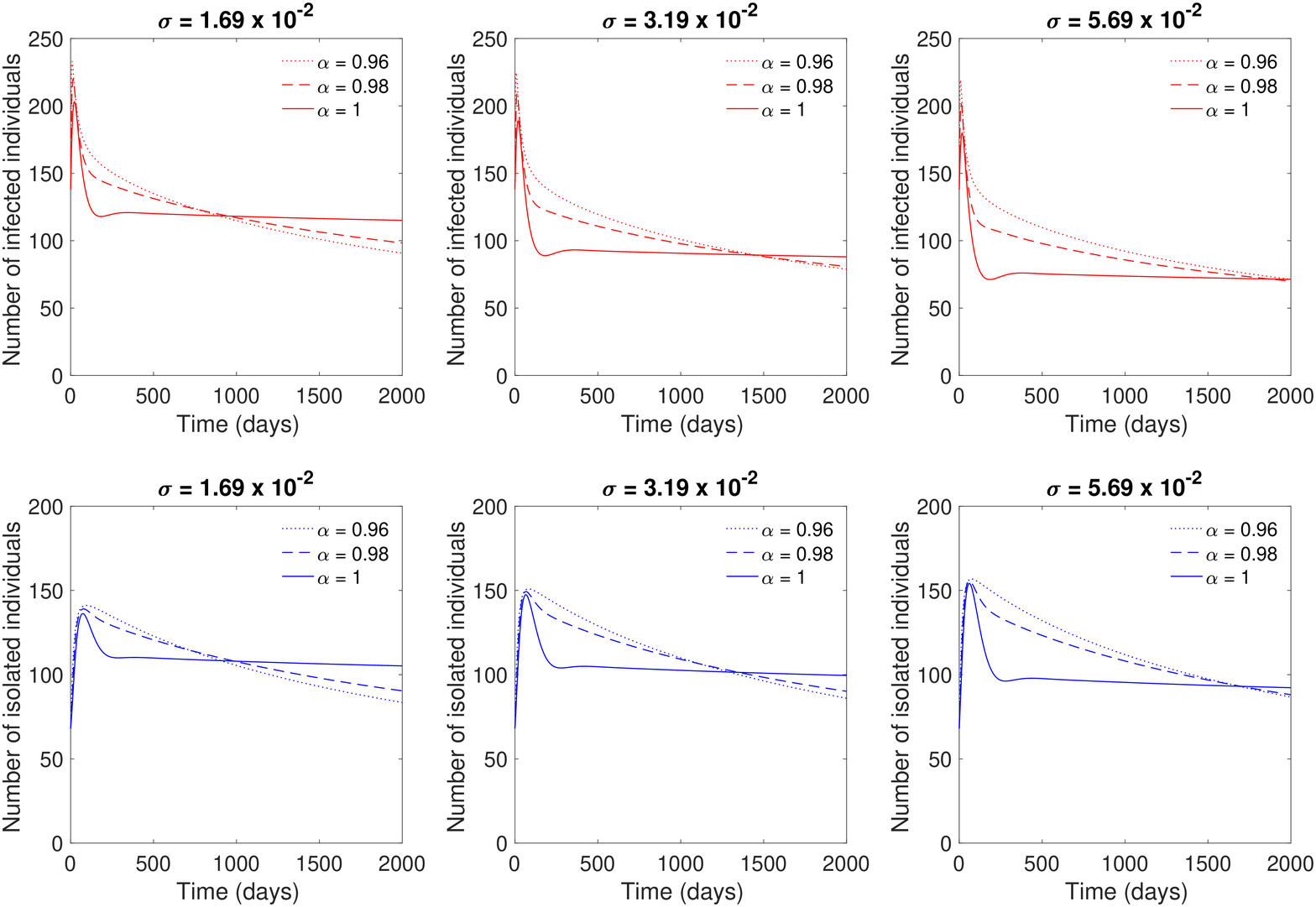}
\caption{Dynamics of infected and isolated individuals. Initial conditions and parameter values are given in \eqref{ini_cond} and in the Table \ref{tabela}, respectively, except $\sigma = \{1.69 \times 10^{-2}, 3.19 \times 10^{-2}, 5.69 \times 10^{-2}\}$, $r=0.30 \beta$ and $\alpha = \{0.96, 0.98, 1\}$.}
\label{sigma}
\end{figure}

\begin{figure}[!]
%\vspace{2.5cm}
\center
\includegraphics[scale=0.29]{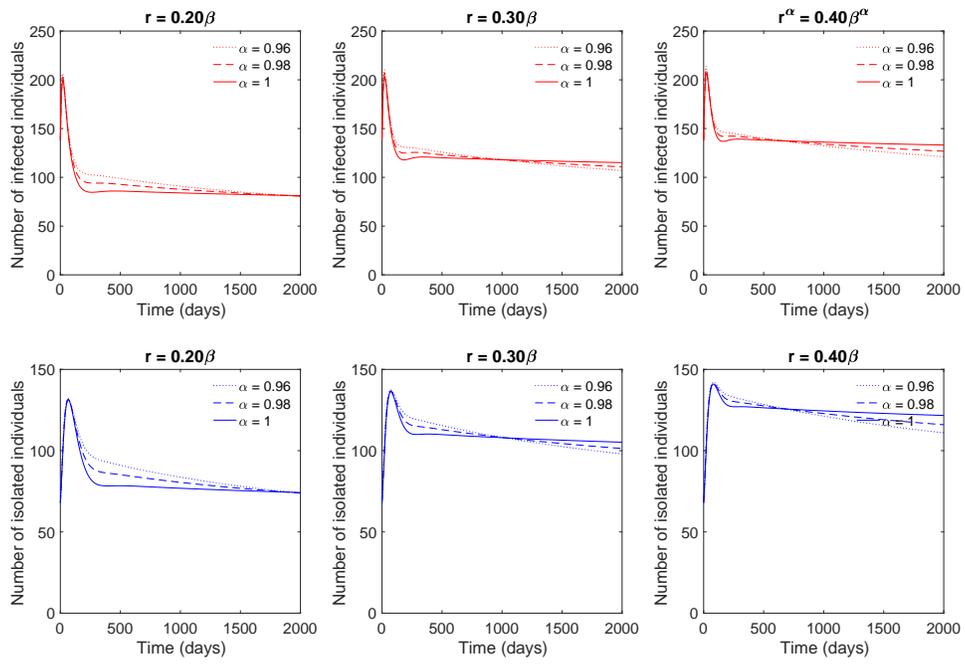}
\caption{Dynamics of infected and isolated individuals. Initial conditions and parameter values are given in \eqref{ini_cond} and in the Table \ref{tabela}, respectively, except $r = \{0.20\beta, 0.30\beta, 0.40\beta \}$ and $\alpha = \{0.96, 0.98, 1\}$.}
\label{r}
\end{figure}

\begin{figure}[!]
%\vspace{2.5cm}
\center
\includegraphics[scale=0.37]{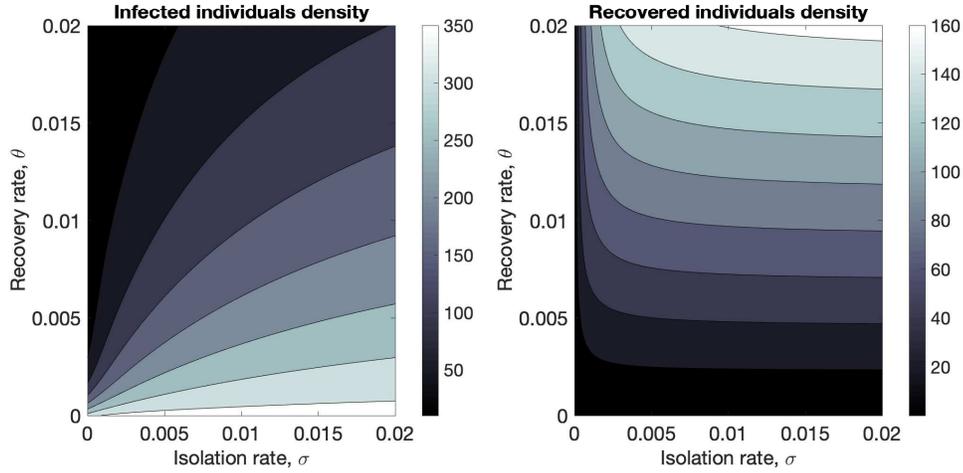}
\caption{Infected (left) and recovered (right) individuals density after 1000 days, considering some combinations of isolation $\sigma$, and recovered $\theta$, rates, for $\alpha = 1$. Initial conditions and parameter values are given in \eqref{ini_cond} and in the Table \ref{tabela}, respectively.}
\label{contour}
\end{figure}

\section{Conclusion} \label{conclusions}

In this work, a FO model for the dynamics of a population in the presence of CoViD-19 was formulated and analyzed.

From a theoretical point of view, the basic reproduction number was calculated and the impact of the parameters of the model was discussed. Local stability around the disease-free equilibrium point was proven for $\mathcal{R}_0 < 1$ (see Lemma \ref{estabilidadeR0}).

From a numerical point of view the model was simulated for relevant parameters. The isolation of people with CoViD-19 leads to a decline in the number of infected people. On the other hand, the lessening of the disease in the population leads to a smaller need for isolation. A decrease of the FO derivative $\alpha$, results in fewer people being infected and isolated over time. We can see these three results through Figure \ref{sigma}. A greater susceptibility to a new infection increases the number of people infected and causes CoViD-19 to firmly persist in the population. So, the higher the rate of reinfection in the population, the more people will become infected. Consequently, the number of people in quarantine will also increase (see Figure \ref{r}). Regardless of the value of the reinfection rate, low $\alpha$ values reflect fewer people with CoViD-19 and fewer people in quarantine, as we can also see in Figure \ref{r}. Moreover, a high recovery rate and an isolation rate above $10^{-2}$ is reflected in a population with fewer patients CoViD-19. Furthermore, a population with a small number of people recovered from the disease cannot reduce the number of infected individuals (see Figure \ref{contour}). Consequently, the disease spreads faster in the population. Although the simulations were performed for a small standard population, it is assumed that the results also apply to larger populations, since the parameter estimation was done on large populations \cite{Read2020}. 

Model analysis and predictions provide crucial data that can be used in treatment strategies and prevention. Moreover, they are the tools needed to demonstrate the impact of a social or behavioural intervention in a population. This could help policy makers to devise strategies to reduce heavy economic and social burden of SARS-CoV-2 infection in the world. In this paper, numerical results illustrate the dynamics of a standard population in the presence of CoViD-19. They also show the effect of reinfection and quarantine on the number of infections. The results suggest that policy makers should consider specific measures to reduce SARS-CoV-2 infection such as: developing campaigns to alert individuals on how to avoid contact with each other, reducing contagion as much as possible; explaining how reinfection can be a decisive factor in increasing the number of infected people; raising media awareness of the importance of quarantine in fighting the pandemic, among others. 

The order of the fractional derivative may provide better fits to real data from patients infected with SARS-CoV-2, as previously seen for other fractional order models for CoViD-19 \cite{JPCarvalho, Ahmad}. The model does not include the impact of vaccination on the population. In our future work we will study population dynamics in the presence of different levels of vaccine efficiency for CoViD19.

\end{document}